\documentclass[11pt,usenames,dvipsnames,reqno]{amsart} 
\usepackage[margin=1.25in]{geometry} 
\usepackage{xypic}
\usepackage{amsmath}
\usepackage{amssymb}
\usepackage{amsthm}
\usepackage[colorlinks=true,linkcolor=NavyBlue,citecolor=NavyBlue,anchorcolor=NavyBlue,urlcolor=NavyBlue]{hyperref}
\usepackage{graphicx}
\usepackage{enumitem}
\usepackage{tikz}
\usepackage{mdframed}

\usepackage{yfonts}
\usepackage[T1]{fontenc}

\newcommand{\rr}{\mathbb{R}}
\renewcommand{\sl}{\mathrm{sl}}
\newcommand{\tb}{\mathrm{tb}}
\newcommand{\st}{{\mathrm{st}}}
\newcommand{\id}{\operatorname{id}}
\newcommand{\pr}{\operatorname{pr}}

\newcommand{\mgraph}{\mathcal{T}}
\newcommand{\wire}{w}
\newcommand{\Skel}{\textsc{Skel}}
\newcommand{\Coskel}{\textsc{Coskel}}

\newtheoremstyle{thm}{15 pt}{10 pt}{\itshape}{}{\bfseries}{.}{.5em}{}
\newtheorem{thm-main}{Theorem}
\newtheorem{cor-main}[thm-main]{Corollary}
\newtheorem{theorem}{Theorem}
\numberwithin{theorem}{section}
\newtheorem{corollary}[theorem]{Corollary}

\newtheorem{question}[theorem]{Question}
\newtheorem{lemma}[theorem]{Lemma}
\newtheorem{conjecture}[theorem]{Conjecture}
\newtheorem{proposition}[theorem]{Proposition}

\newtheorem*{thm*}{Theorem}
\newtheorem*{cor*}{Corollary}

\newtheoremstyle{ex}{10 pt}{10 pt}{\normalfont}{}{\bfseries}{.}{.5em}{}
\theoremstyle{ex}
\newtheorem{example}[theorem]{Example}
\newtheorem{definition}[theorem]{Definition}

\newtheoremstyle{rem}{10 pt}{5 pt}{\normalfont}{}{\it}{.}{.5em}{}
\theoremstyle{rem}
\newtheorem{remark}[theorem]{Remark}

\numberwithin{equation}{section}

\usepackage{color}

\begin{document}

\thispagestyle{empty}
\title[Legendrian ribbons and strongly quasipositive links]{Legendrian ribbons and strongly quasipositive links in an open book}
\author[K. Hayden]{Kyle Hayden} \address{Boston College, Chestnut Hill, MA 02467} \email{kyle.hayden@bc.edu}

\maketitle

\begin{abstract}  
We show that a link in an open book can be realized as a strongly quasipositive braid if and only if it bounds a Legendrian ribbon with respect to the associated contact structure.  This generalizes a result due to Baader and Ishikawa for links in the three-sphere. We highlight some related techniques for determining whether or not a link is strongly quasipositive, emphasizing applications to  fibered links and satellites.    
 \end{abstract}

\section{Introduction} \label{sec:intro}

Among all common families of links in $S^3$, the strongly quasipositive braids defined by Rudolph \cite{rudolph:kauffman-bound} are perhaps the most intimately connected to contact geometry.  For example, Hedden \cite{hedden:pos} showed that a fibered link in $S^3$ is strongly quasipositive if and only if its open book supports the tight contact structure on $S^3$.  These so-called tight fibered knots include algebraic knots and all knots admitting lens space (or even L-space) surgeries \cite{ni:fibered,hedden:pos}, and it is conjectured that tight fibered knots are linearly independent in the concordance group \cite{baker:tight-fibered, rudolph:concordance}.  More generally, strongly quasipositive links achieve equality in the Bennequin and slice-Bennequin bounds on Seifert and slice genera, and it is conjectured that they are the \emph{only} links for which the Bennequin bound is sharp.

A more elementary connection to contact geometry was established by Baader and Ishikawa in \cite{bi:leg-qp}: the strongly quasipositive links are precisely those links that bound Legendrian ribbons in $(S^3,\xi_\st)$, where a \textbf{Legendrian ribbon} in a contact 3-manifold $(Y,\xi)$ is a surface $R$ that retracts onto some Legendrian graph $\Lambda \subset R$ under a flow tangent to $\xi \cap TR$. These surfaces may be of independent interest for their role in Giroux's correspondence \cite{giroux:geometry} and the fact that they minimize both Seifert and slice genera in $(S^3,\xi_\st)$. 

Recently, the notion of strong quasipositivity was extended to braids in arbitrary open books \cite{square,hayden:stein,hedden:subcritical,ik:bennequin}. Just as in the classical setting, this notion exhibits a close connection to contact geometry.  Our main theorem generalizes the result of Baader and Ishikawa to arbitrary contact 3-manifolds and compatible open books.

\begin{thm-main}\label{thm:sqp-ribbon}
A link in an open book is strongly quasipositive if and only if it bounds a Legendrian ribbon in the associated contact structure.
\end{thm-main}

The forward implication uses a version of Giroux flexibility for surfaces with transverse boundary. The reverse implication follows an approach inspired by Baader and Ishikawa, made possible by Gay and Licata's development of Morse structures on open books \cite{g-l:morse}. As an immediate consequence of Theorem~\ref{thm:sqp-ribbon}, we see that the condition of strong quasipositivity is fundamentally contact-geometric and independent of the choice of compatible open book.

\begin{remark}
After distributing a draft of this paper, I learned that Theorem~\ref{thm:sqp-ribbon} had also been proven in a different manner by Baykur, Etnyre, Hedden, Kawamuro, and Van Horn-Morris. Their proof, obtained independently, will appear in forthcoming joint work.
\end{remark}

We apply these ideas in several directions.  First, we can use Legendrian ribbons to obtain strong constraints on genera of surfaces bounded by strongly quasipositive links. The most basic result in this direction,  observed in \cite{ik:bennequin,hayden:stein}, is that strongly quasipositive links achieve equality in the Bennequin-Eliashberg bound; see \S\ref{subsec:genus}.  This is applied, for example, to study cables of strongly quasipositive knots (in Corollary~\ref{cor:cables}). In \S\ref{subsec:destab}, we study the effect of braid operations on other  notions of 
quasipositivity using a more powerful constraint, one that generalizes a result due to Rudolph for knots in $S^3$: 
\emph{If $(Y,\xi)$ has a symplectic filling $X$ with vanishing second homology, then a nontrivial strongly quasipositive knot in $Y$ cannot bound a slice disk in $X$.}

 Second, we can use Legendrian ribbons to certify that a given link can be represented by a strongly quasipositive braid. This is particularly useful for studying satellites, where a familiar procedure for producing Seifert surfaces of satellite knots is  compatible with Legendrian ribbons. 

\begin{thm-main}\label{thm:satellite}
If $J \subset S^1 \times D^2$ is a strongly quasipositive braid and $K \subset (Y,\xi)$ is a strongly quasipositive  link, then the  satellite $J(K) \subset (Y,\xi)$ is also a strongly quasipositive  link.
\end{thm-main}

\noindent This approach can also be used to demonstrate the strong quasipositivity of sufficiently twisted Whitehead doubles (Example~\ref{ex:doubles}) and of patterns and companions of fibered strongly quasipositive knots in a tight contact manifold (Theorem~\ref{thm:fibered} and Corollary~\ref{cor:satellite}).

Third, motivated by the importance of tight fibered knots in $S^3$, we can ask about the relationship between strong quasipositivity and fiberedness for knots in a tight contact manifold $(Y,\xi)$. Theorem~\ref{thm:sqp-ribbon} implies that a fibered transverse knot is strongly quasipositive with respect to any open book supporting the same contact structure; see \S\ref{subsec:fibered}. It is illuminating to merge this observation with a result  of  Etnyre and Van Horn-Morris:

\begin{cor-main}[cf.~\cite{hedden:pos,E-VHM:fibered}] \label{cor:fibered}
Let $K$ be a fibered link in a tight contact manifold $(Y,\xi)$ with zero Giroux torsion supported by an open book $(B,\pi)$. The following are equivalent:
\begin{enumerate}
\item $K$ is transversely isotopic to a strongly quasipositive braid with respect to $(B,\pi)$;
\item $K$ bounds a Legendrian ribbon in $(Y,\xi)$;
\item the contact structure $\xi_K$ supported by $K$ is isotopic to $\xi$;
\item \label{item:bennequin} $K$ achieves equality in the Bennequin-Eliashberg bound.
\end{enumerate}
\end{cor-main}

Finally, Theorem~\ref{thm:sqp-ribbon} reveals the weakness  of strong quasipositivity in an overtwisted contact structure. Indeed, a result of Baader, Cieliebak, and Vogel  says that \emph{every} nullhomologous link type in an overtwisted contact structure bounds a Legendrian ribbon. 

\begin{cor-main}[cf. \cite{bcv:ribbons}]
In an overtwisted contact structure, every nullhomologous link type is strongly quasipositive. \hfill $\square$
\end{cor-main}

This corollary serves as a dramatic demonstration of the principle that strong quasipositivity is fundamentally contact-geometric.

\subsection*{Organization} We recall the necessary background material in \S\ref{sec:prelim}, with an emphasis on Morse structures on open books and front projections of Legendrian ribbons. The main theorem and applications are proven in \S\ref{sec:ribbon-qp} and \S\ref{sec:applications}, respectively.

\subsection*{Acknowledgements} I wish to thank Sebastian Baader for an inspiring correspondence about the results of \cite{bi:leg-qp} and \cite{bcv:ribbons}, as well as John Baldwin, John Etnyre,  David Gay, and Eli Grigsby for helpful conversations. Thanks also to the members of the AIM SQuaRE on contact and symplectic geometry and the mapping class groups (\cite{square}) for their support.

\section{Preliminaries}\label{sec:prelim}

\subsection{Open books and quasipositivity} A celebrated theorem of Giroux \cite{giroux:geometry} (building on work of Thurston-Winkelnkemper \cite{tw:existence}) establishes a one-to-one correspondence between open books and contact structures on any closed, orientable 3-manifold $Y$; see \cite{etnyre:obd} for a thorough account. There is an analogous correspondence between links that are braided with respect to the pages of an open book and links that are transverse to the planes of a contact structure; see, for example, \cite{bennequin,pavelescu:thesis,os:markov}. We also see Giroux's correspondence reflected in the relationship between singular foliations on embedded surfaces in $Y$ induced by open books and by contact structures; see \cite{ito-kawamuro:open} for background.

A \textbf{Bennequin surface} for a braid in an open book $(F,\varphi)$ is a Seifert surface formed from disks and half-twisted bands, where the disks are meridional for the binding and each band is attached along an embedded arc $\alpha$ lying in a single page $F_\theta$; Figure~\ref{fig:bennequin} depicts an example where the page is a half-plane. The band has exactly one point where it is tangent to the page $F_\theta$, and we say the band is \textbf{positive} or \textbf{negative} according to whether the orientation of the Seifert surface agrees or disagrees, respectively, with the orientation of the page at this tangency. Equivalently, we can define Bennequin surfaces in terms of their open book foliations, which consist only of $aa$-tiles; see the definitions and background in \cite{ik:bennequin}.

\begin{definition} \label{def:sqp}
A braid in an open book is \textbf{strongly quasipositive} if it bounds a Bennequin surface with only positive bands.
\end{definition}

We can extend this definition to transverse links using the correspondence discussed above so that a transverse link is strongly quasipositive if it is transversely isotopic to a strongly quasipositive braid with respect to some compatible open book.

\begin{figure}\center
\def\svgwidth{300pt} 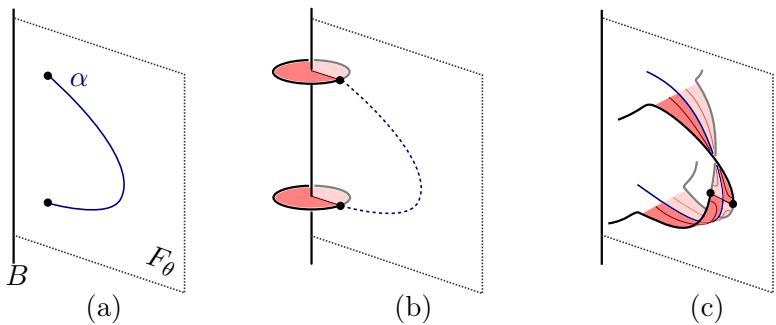 
\caption{Building a Bennequin surface along a core arc $\alpha$ in a page $F_\theta$.} \label{fig:bennequin}
\end{figure}

\subsection{Morse structures on open books} \label{subsec:morse} 
Here we collect the necessary definitions and results from Gay and Licata's study of Morse structures on open books; all material in this subsection is drawn from \cite{g-l:morse}. Define a contact 3-manifold $(W_0,\xi_0=\ker \alpha_0)$ by
$$W_0=(0, \infty) \times S^1 \times S^1, \quad \alpha_0 = dz + x \, dy,$$
where $x,y,z$ are coordinates on the three factors of $W$.

\begin{theorem}[Gay-Licata]\label{thm:contacto}
Let $Y$ be a closed 3-manifold with open book $(B,\pi)$. Then there is a contact structure $\xi$ on $Y$ compatible with the open book and a 2-complex $\Skel \subset Y \setminus B$ such that every connected component of $(Y \setminus (\Skel \cup B),\xi)$ is contactomorphic to $(W_0,\xi_0)$. \end{theorem}

Under the contactomorphism in Theorem~\ref{thm:contacto}, we can identify the boundary of a neighborhood of each component of $B$ with $\{1\} \times S^1 \times S^1$. The contactomorphism is constructed using the flow of a certain vector field $V$ on $Y \setminus B$; see \cite[Definition~1.2]{g-l:morse} for details. For our purposes, it suffices to note that the restriction of $V$ to each page $F_\theta$ is a vector field $V_\theta$ that has a single source, no sinks, and is tangent to the characteristic foliation.

The \textbf{skeleton} $\Skel\subset Y \setminus B$ is the union over all pages of the descending manifolds of the index-1 singularities of $V_\theta$, together with the index-0 singularity on each page. The \textbf{co-skeleton}, denoted $\Coskel$, is the union over all pages of the ascending manifolds of the index-1 singularities. The 2-complex $\Coskel$ intersects the boundary $\amalg^n S^1 \times S^1$ of a regular neighborhood of the binding $B$ in a trivalent graph $\mgraph$, and the isotopy type of this graph determines the original  open book $(Y,B,\pi)$ up to diffeomorphism. This collection of decorated tori is called a \textbf{Morse diagram} for $(Y,B,\pi)$. 

Morse diagrams can also be considered abstractly: An \textbf{abstract Morse diagram} is a collection of tori $\amalg^n S^1 \times S^1$ with a finite trivalent graph $\mgraph$ such that
\begin{enumerate}[label=(\roman*)]
\item the edges of $\mgraph$ are monotonic with respect to the second $S^1$ factor;
\item for each fixed value $c$ of the second factor, there is a pairing on curves intersecting $\amalg^n S^1 \times \{c\}$, and the pairing is constant away from vertices;
\item surgery on $\amalg^n S^1 \times \{c\}$ with attaching spheres given by paired points on the curves yields a single $S^1$; and
\item trivalent points occur in pairs on the same slice $\amalg^n S^1 \times \{c\}$. As $t \to c^-$, a curve labeled $x$ approaches a curve labeled $y$ from the left (respectively, right), while as $t \to c^+$, a curve labeled $x$ approaches the other $y$-curve from the right (left).
\end{enumerate}

\begin{example}\label{ex:morse}
Three abstract Morse diagrams are shown in Figure~\ref{fig:ex-diagram}, adapted from \cite[Figure 1]{g-l:morse} to illustrate the difference in our conventions: \textbf{(a)} An annular open book whose monodromy consists of two right-handed Dehn twists about the core curve, yielding the lens space $L(2,1)$ with its universally tight contact structure. \textbf{(b)} An open book with once-punctured torus page and monodromy consisting of a left-handed Dehn twist along a nonseparating curve and a right-handed Dehn twist along boundary-parallel curve. \textbf{(c)} An annular open book whose monodromy is a single left-handed Dehn twist about the core curve, yielding an overtwisted contact structure on $S^3$.
\end{example}

\begin{figure} \center
\def\svgwidth{\linewidth}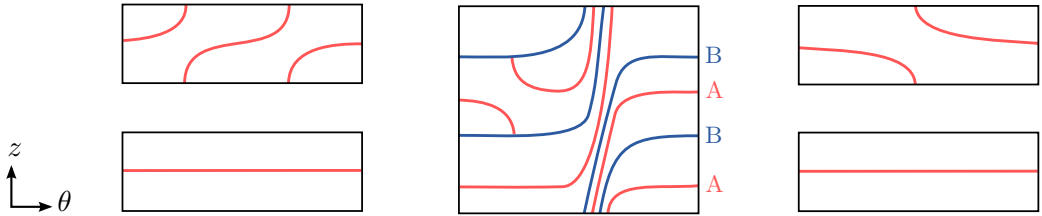 
\caption{Abstract Morse diagrams from Example~\ref{ex:morse}.}
 \label{fig:ex-diagram}
\end{figure}

\subsection{Front projections} The identification of the connected components of $Y \setminus (\Skel \cup B)$ with $W_0 \cong (0,\infty) \times S^1 \times S^1$ allows us to define a \textbf{front projection} 
$$\pr: Y \setminus (\Skel \cup B) \to \amalg^n S^1 \times S^1.$$
Here we view each torus in $\amalg^n S^1 \times S^1$ as $\{1\} \times S^1 \times S^1$, the boundary of a regular neighborhood of a binding component. 

A generic smooth knot $K$ in $Y$ will avoid the binding $B$ and intersect the 2-complex $\Skel$ transversely. We can assume that the image of the knot under the projection map to the collection of tori is an immersion with transverse double points, forming the starting point for a knot diagram. The key difference is that the image will generally consist of immersed arcs whose endpoints lie on the trivalent graph $\mgraph$ inside the Morse diagram. The endpoints of these arcs correspond to intersections of $K$ with $\Skel$. When the knot $K$ is Legendrian (i.e. when $TK$ lies inside $\xi$), the resulting knot diagram exhibits many of the same features as Legendrian fronts in $(\rr^3,\xi_\st)$. See Figure~\ref{fig:front} for examples.

\begin{definition}\label{def:front}
A \textbf{front} on a Morse diagram $(\amalg^n S^1 \times S^1,\mgraph)$ is a collection of arcs and closed curves ${D}$ immersed, with semicubical cusps, in $\amalg^n S^1 \times S^1$, satisfying the following properties:
\begin{enumerate}[label=(\roman*)]
\item The slopes at all interior points on ${D}$ are negative (using coordinates $(\theta,z)$ on $S^1 \times S^1$ and measuring slope as $dz/d\theta$).
\item The endpoints of arcs of ${D}$ lie on the interiors of curves of $\mgraph$ and have slope 0.
\item Suppose that $e$ and $e'$ are two edges of $\mgraph$ with the same label. For every arc of ${D}$ ending on $e$ at height $t$, approaching $e$ from the left (respectively, right), there is an arc of ${D}$ ending on $e'$ at height $t$, approaching $e'$ from the right (left).
\end{enumerate}
\end{definition}

\begin{figure} \center
\def\svgwidth{\linewidth}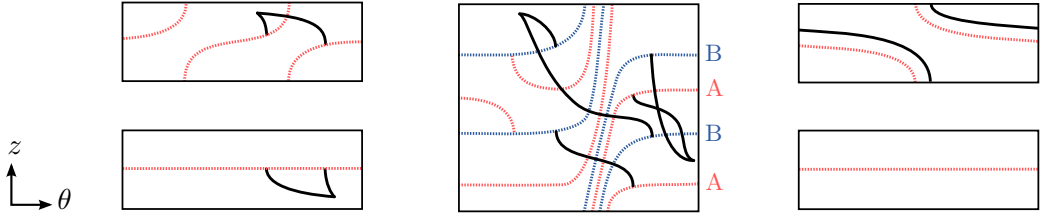 
\caption{Legendrian fronts in Morse diagrams.}
 \label{fig:front}
\end{figure}

The following theorem says that we can use front projections in Morse diagrams in much the same manner as front projections in $(\mathbb{R}^3,\xi_\st)$.

\begin{theorem}[Gay-Licata]\label{thm:front}
Let $\Lambda$ be a Legendrian link in $(Y,\xi)$ that is disjoint from the binding and transverse to $\Skel$. Then the image of $\Lambda$ under the flow by $\pm V$ to $\amalg^n S^1 \times S^1$ is a front on the Morse diagram. Furthermore, any front on this Morse diagram is the image of such a Legendrian $\Lambda$, and any two Legendrians with the same front are equal.
\end{theorem}

Gay and Licata also provide a set of diagram moves for Legendrian fronts and prove an analogue of Reidemeister's theorem. The preceding definition and theorem extend naturally to allow for front projections of \textbf{Legendrian graphs}, embedded spatial graphs whose edges are tangent to the contact planes.  The set of Reidemeister moves for Legendrian graphs in $(\rr^3,\xi_\st)$ from \cite{bi:leg-qp} also extends in the obvious way. In addition to these Reidemeister moves, we also wish to consider \textbf{subdivision} of the edges of a Legendrian graph, which consists of placing a new vertex along an existing edge.

\subsection{Legendrian ribbons} As discussed in \S\ref{sec:intro}, we define a \textbf{Legendrian ribbon} of a Legendrian graph $\Lambda$ in $(Y,\xi)$ to be a smoothly embedded surface  $R \subset Y$ with transverse boundary such that 
\begin{enumerate}
\item $\Lambda$ is in the interior of $R$ and $R$ retracts onto $\Lambda$ under a flow tangent to $\xi|_R$,
\item for each $p \in \Lambda$, the 2-plane $\xi_p$ is tangent to $R$, and
\item for each $p \in R \setminus \Lambda$, the 2-plane $\xi_p$ is transverse to $R$.
\end{enumerate}

\smallskip

\noindent In $(\rr^3,\xi_\st)$, there is a standard recipe for building a Legendrian ribbon using front diagrams. One begins by  replacing the cusps and vertices of the diagram with ribbon neighborhoods as depicted on the left half of Figure~\ref{fig:recipe}. Away from these cusps and vertices, we extend the ribbon in such a way that it undergoes one half-twist along each segment that contains no cusps or vertices. For a thorough justification of these steps, see \cite[Algorithm 2]{avdek}. Ambiently, we can understand the presence of the half-twist as follows: First, construct a band by pushing the arc off itself along the $\pm \partial/\partial y$ directions. The front projection of this band is degenerate and will have the same image as the arc itself. Thus, to ensure that the band can glue to the local ribbon neighborhoods of the cusps/vertices at its ends, we apply a quarter-twist to each end of the band. These combine to form a single half-twist as depicted.

\begin{figure} \center
\includegraphics[width=\linewidth]{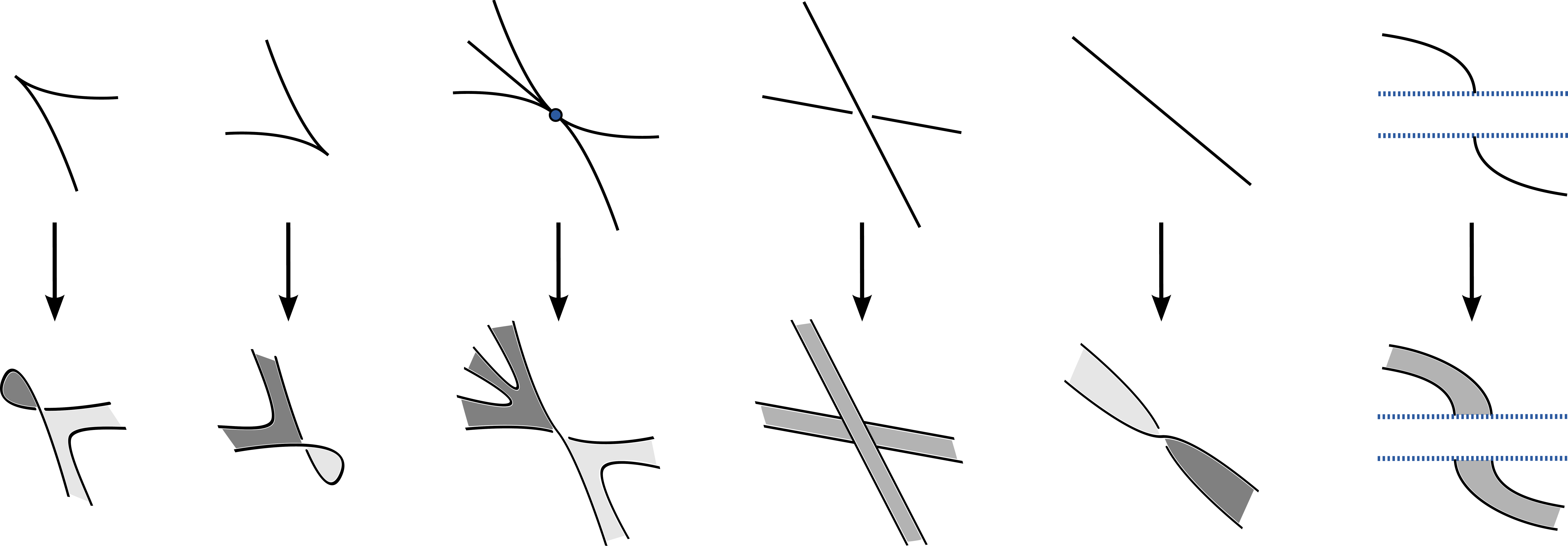} 
\caption{Building a ribbon from a Legendrian front.}
 \label{fig:recipe}
\end{figure}

When considering Legendrian fronts in a Morse diagram, the procedure for producing front projections of Legendrian ribbons is nearly identical. Suppose that $\Lambda \subset (Y,\xi)$ is a Legendrian graph with front projection $D(\Lambda)$ inside  an associated Morse diagram $(\amalg^n S^1 \times S^1,\mgraph)$. We can assume that all cusps and vertices of $D(\Lambda)$ occur away from $\mgraph$, so the local moves from Figure~\ref{fig:recipe} generalize in the obvious way. Again we extend the ribbon over each cusp-/vertex-free segment of the diagram, introducing a single half-twist along each such segment. This choice is justified just as in the standard setting; now we are pushing the arc off itself using the vector field $V$ instead of $\partial/\partial y$. For visual clarity, we can assume the twist occurs away from $\mgraph$ as shown in  Figure~\ref{fig:recipe}.

It is easy to check that subdivision of edges preserves the isotopy type of the Legendrian ribbon and the transverse isotopy type of its boundary.

\section{From Bennequin surfaces to Legendrian ribbons and back} \label{sec:ribbon-qp}

Our goal in this section is to prove the main theorem, which draws an equivalence between strongly quasipositive links in an open book and links bounding Legendrian ribbons in a compatible contact structure.

\subsection{From Bennequin surfaces to Legendrian ribbons}

We begin by showing that every strongly quasipositive transverse link in an open book bounds a Legendrian ribbon. Though this observation was included in \cite{hayden:stein}, we reprove it here for convenience. The following criterion helps us identify when a Seifert surface is isotopic to a Legendrian ribbon.

\begin{lemma}[Ribbon flexibility]\label{lem:recognize}
A transverse link is the boundary of a Legendrian ribbon if and only if it has a Seifert surface whose characteristic foliation is Morse-Smale and contains no negative singularities or attracting closed leaves.
\end{lemma}

\begin{proof}
First, suppose that $R$ is a Legendrian ribbon in $(Y,\xi)$. 
After a small perturbation, we can assume that the characteristic foliation on $R$ is Morse-Smale, contains no closed leaves, and has a positive elliptic point at each vertex and a single positive hyperbolic point along each edge, as claimed. 

Conversely, suppose that $K$ is a transverse link with a Seifert surface $S$ whose characteristic foliation is Morse-Smale and contains no attracting closed leaves or negative singularities. We may eliminate any repelling closed leaf $\ell$ of the characteristic foliation by introducing a positive elliptic-hyperbolic pair along $\ell$ as in Figure~\ref{fig:creation}; see Lemma~2.3 of \cite{ef:trivial}.  Now let $\Gamma$ be the Legendrian graph obtained as (the closure of) the union of the hyperbolic points' stable manifolds, and let $S_+$ denote a small neighborhood of $\Gamma$ in $S$. By the Morse-Smale condition and the fact that there are no negative singularities or closed leaves, the subsurface $S \setminus S_+$ is a collection of annuli whose characteristic foliation consists of parallel arcs running from $\partial S$ to $\partial S_+$. Therefore  the characteristic foliation on $S$ is divided by a set of core curves for these annuli, one parallel to each boundary component of $S$.  As in \cite[Lemma 2.1]{E-VHM:fibered}, we can construct a contact vector field transverse to $S$ with this dividing set. It is easy to see that there is another singular foliation $\mathcal{F}'$ on $S$ that is conjugate to the characteristic foliation on a Legendrian ribbon of $\Gamma$ and agrees with the characteristic foliation on $S$ outside of $S_+$. The dividing set constructed above also divides the singular foliation $\mathcal{F}'$. An appropriate version of Giroux's flexibility theorem (see \cite[Theorem~2.2]{E-VHM:fibered}) now says that we can isotope $S$ rel boundary so that its characteristic foliation agrees with $\mathcal{F}'$. It follows that there is an isotopy carrying $S$ to a Legendrian ribbon, and this isotopy preserves the transverse isotopy type of the boundary.
\end{proof}

\begin{figure} \center
\def\svgwidth{325pt}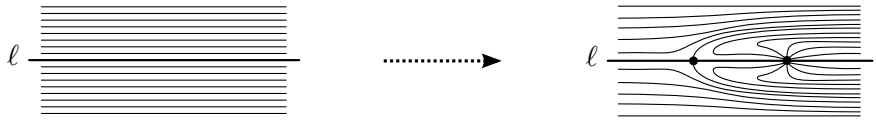 
\caption{Creation of an elliptic-hyperbolic pair in the characteristic foliation.}
 \label{fig:creation}
\end{figure}

\begin{proposition}\label{prop:qp-to-r}
A Bennequin surface with only positive bands in an open book is isotopic to a Legendrian ribbon with respect to the compatible contact structure. Moreover, the isotopy restricts to transverse isotopy along the boundary.
\end{proposition}

\begin{proof}
The open book foliation on the Bennequin surface consists of positive $aa$-tiles. By Theorem~2.21 of \cite{ito-kawamuro:open}, we can isotope $\xi$ so that the characteristic foliation and open book foliation are conjugate. The claim now follows from the lemma above. \end{proof}

\subsection{From Legendrian ribbons to Bennequin surfaces} To complete the proof of Theorem~\ref{thm:sqp-ribbon}, we establish the following converse to Proposition~\ref{prop:qp-to-r}:

\begin{proposition}\label{prop:r-to-qp}
A Legendrian ribbon in $(Y,\xi)$ can be isotoped to a Bennequin surface with only positive bands with respect to any compatible open book. Moreover, the isotopy restricts to transverse isotopy along the boundary.
\end{proposition}

We will show that any Legendrian graph in an open book can be arranged in a position analogous to the familiar ``arc presentations'' of knots and links in $\mathbb{R}^3$; see the survey \cite{cromwell:survey}. When the graph is arranged in this manner, its Legendrian ribbon is naturally isotopic to a Bennequin surface. We first illustrate the strategy with an example.

\begin{example}\label{ex:ribbon} The set of diagrams in Figure~\ref{fig:ribbon-ex} above depicts the construction of a special Legendrian ribbon for the Legendrian knot from the middle of Figure~\ref{fig:front}. In the first step, vertices are added at the cusps and at some intermediate points. As a result, the middle diagram consists of segments that are almost horizontal near vertices and nearly vertical elsewhere. The resulting front projection for the Legendrian ribbon consists of disks joined by positively twisted bands with braided boundary. The back of each disk consists of a negatively braided segment that can be made positively braided by a transverse isotopy that pushes the segment through the binding. 

\begin{figure} \center
\def\svgwidth{\linewidth}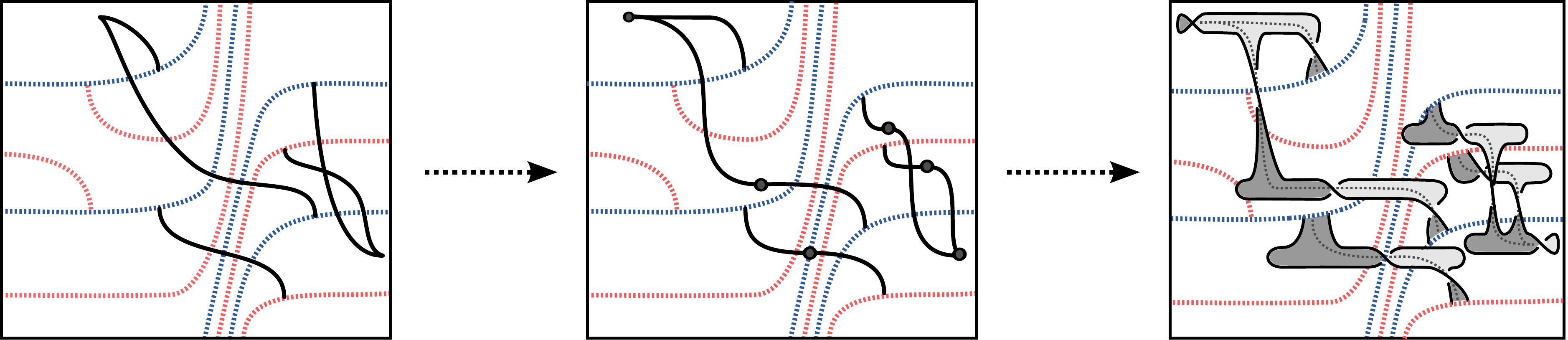 
\caption{Arranging a Legendrian graph so that its ribbon is nearly a Bennequin surface.}
 \label{fig:ribbon-ex}
\end{figure}
\end{example}

We define an \textbf{arc presentation} of a graph $G$ in an open book $(B,\pi)$ to be an embedding of $G$ in a finite collection of pages such that all vertices lie on the binding and every edge is a single simple arc in a page $F_\theta=\pi^{-1}(\theta)$. Given a Morse structure on $(B,\pi)$, we can assume that $G$ is transverse to $\Skel$. Consider the front projection of $G$. Each edge projects to a union of disjoint arcs $\wire = \cup_i \, \wire_i$ in $(\amalg^n S^1 \times S^1,\mgraph)$ such that all points in $\wire$ have the same $\theta$-coordinate and all but two points in $\partial \wire$ lie on $\mgraph$. The points in $\partial \wire$ that lie on $\mgraph$ can be paired as in condition (iii) of Definition~\ref{def:front}. We call such a union of arcs $\wire$ a \textbf{wire} in the Morse diagram, and we refer to the two points in $\partial \wire \setminus \mgraph$ as the \textbf{ends} of the wire. 
We refer to the front projection of $G$, i.e. the collection of wires, as an \textbf{arc diagram} for $G$.

Though an arc diagram does not satisfy conditions (i) and (ii) of Definition~\ref{def:front}, we can modify it so that it does correspond to the front projection of a Legendrian graph: For sufficiently small $\epsilon>0$, we can use a small isotopy of $\amalg^n S^1 \times S^1$ that preserves each $S^1 \times \{z\}$ to approximate any wire $\wire$ by a union $\wire'$ of arcs such that the slope at any point lies in $[-\infty,0]$ and is
\begin{enumerate}
\item infinite at the points of $\partial \wire \cap \mathcal{T}$,
\item zero at the ends of the wire, and
\item equal to $-1/\epsilon$ away from $\partial \wire$.
\end{enumerate}
We call this a \textbf{cusped arc diagram}, and we continue to refer to the unions of arcs as ``wires''. By (the proof) of Theorem~\ref{thm:front}, each wire  in the cusped arc diagram lifts to a Legendrian arc in $(Y,\xi)$ with its endpoints on the binding $B$. Since every arc diagram determines a cusped arc diagram that is unique up to isotopy through cusped arc diagrams, applying Theorem~\ref{thm:front} yields the following:

\begin{proposition}
Every arc diagram in an abstract Morse diagram for an open book decomposition $(B,\pi)$ for $(Y,\xi)$ determines a Legendrian graph, unique up to Legendrian isotopy of the edges. \hfill $\square$
\end{proposition}

If the front projection of a Legendrian graph is a cusped arc diagram, then we say that the Legendrian graph is in \textbf{cusped arc position} with respect to the open book and Morse structure. 

\begin{proposition}\label{prop:arc-pos}
After subdivision of edges, every Legendrian graph in $(B,\pi)$ can be isotoped to lie in cusped arc position.
\end{proposition}

\begin{proof} By an isotopy through Legendrian graphs, we may assume that the vertices of our graph $\Lambda$ lie on the binding $B$ and that the front projection $D(\Lambda)$ is a generic front satisfying (i)-(iii) of Definition~\ref{def:front}, with the exception that there is a collection of open or half-open arcs in $D(\Lambda)$ whose slope approaches zero; these correspond to where the edges approach the vertices on the binding. Note that the vertices of $\Lambda$ are no longer visible in the front projection.

Define a \textbf{slanted rectangular graph} in $S^1 \times S^1$ to be an embedded graph whose edges all have slope $-\epsilon$ or $-1/\epsilon$ for a small value $\epsilon>0$. Note that a vertex in such a graph has valence at most four. We say that a valence-two vertex is a \textbf{cusp} if the incident edges form an acute angle. Choose a slanted rectangular approximation $G$ to the Legendrian front $D(\Lambda)$ such that 
\begin{enumerate}
\item the cusps of $G$ and $D(\Lambda)$ agree,
\item the intersections $G \cap \mgraph$ and $D(\Lambda) \cap \mgraph$ agree, and
\item edges of $G$ meet $\mgraph$ with slope $-1/\epsilon$.
\end{enumerate}
By smoothing the non-cusp vertices of valence two in $G$ and perturbing $G$ near valence-one vertices on $\mgraph$ so that the incoming edge has infinite slope, we obtain a front for a Legendrian graph $\Lambda_G$. Up to subdivision, the Legendrian graphs $\Lambda$ and $\Lambda_G$ are isotopic. By subdividing each edge of slope $-\epsilon$ in $G$ and contracting towards the binding, we obtain a cusped arc diagram whose lift is equivalent to $\Lambda$.
\end{proof}

Note that, as a special case, we obtain a generalization of the classical fact that every link in $S^3$ admits an arc presentation with respect to the standard open book.

\begin{corollary}\label{cor:arc-pres}
Every link in an open book admits an arc presentation. \hfill $\square$
\end{corollary}

The claimed connection between Legendrian ribbons and positive Bennequin surfaces follows easily from this setup.

\begin{proof}[Proof of Proposition~\ref{prop:r-to-qp}]
Let $R$ be a Legendrian ribbon of the Legendrian graph $\Lambda$. By Proposition~\ref{prop:arc-pos}, we can place $\Lambda$ in cusped arc position after subdividing edges. Subdivision of edges and isotopy of Legendrian graphs not only preserve the smooth isotopy type of $R$ but also the transverse isotopy type of  its boundary.

\begin{figure} \center
\includegraphics[width=\linewidth]{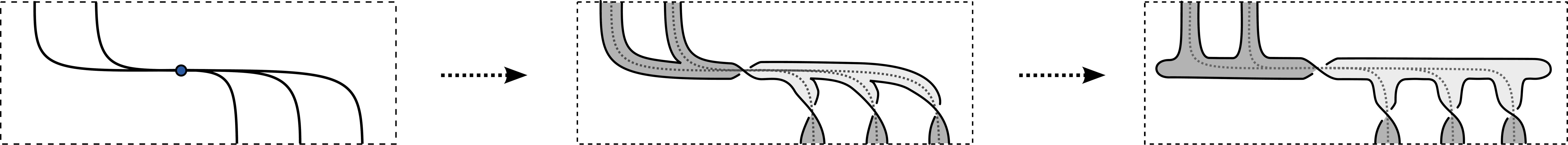} 
\caption{Forming the Legendrian ribbon near a vertex, followed by transverse isotopy of the boundary.}
 \label{fig:vertex}
\end{figure}

Now perturb $\Lambda$ slightly so that it once again misses the binding. Away from the binding, the edges of $\Lambda$ still have vanishingly small $\theta$-support. In the front projection, every vertex $v$ has a neighborhood as depicted in Figure~\ref{fig:vertex}; of course, the number and configuration of edges at $v$ will vary. As in that figure, we form the front projection of the Legendrian ribbon $R$, placing a single half-twist along each downward edge as shown. Away from these neighborhoods, the edges and ribbon are nearly vertical. By a slight isotopy of the surface induced by a transverse isotopy of the boundary, we can arrange so that almost all of $\partial R$ is (positively) braided. This is indicated in the second step of Figure~\ref{fig:vertex}. The only non-braided portions of $\partial R$ are the arcs that wind around the back of the disk at each vertex. We braid each such arc by using a transverse isotopy to push it back across the binding. The ribbon $R$ is now a Bennequin surface with positively twisted bands, as desired.
\end{proof}

\section{Applications}\label{sec:applications}

\subsection{Genus bounds} \label{subsec:genus}
 The majority of our applications will invoke some constraint on the topology of a surface bounded by a strongly quasipositive knot, so we begin by collecting these constraints. Most fundamentally, we have the Bennequin-Eliashberg inequality \cite{bennequin,yasha:20yrs}: If $L$ is a nullhomologous transverse link in a tight contact 3-manifold with Seifert surface $\Sigma$, then the self-linking number of $L$ satisfies
\begin{equation}\label{eq:be}
\sl(L,[\Sigma])\leq - \chi(\Sigma).
\end{equation}

As mentioned in \S\ref{sec:intro}, it has been conjectured that strongly quasipositive braids are the only braids for which the Bennequin-Eliashberg bound is sharp; see \cite{ik:bennequin}. 

\begin{conjecture}\label{conj:bennequin}
A transverse link $K$ in a tight contact 3-manifold $(Y,\xi)$ achieves sharpness in the Bennequin-Eliashberg bound if and only if $K$ is strongly quasipositive.
\end{conjecture}

In light of Theorem~\ref{thm:sqp-ribbon}, this conjecture is equivalent to showing that $L$ achieves equality in \eqref{eq:be} with respect to a given relative homology class if and only if that relative homology class contains a Legendrian ribbon. The reverse implication is a simple consequence of the definition of the transverse self-linking number, so (as observed in \cite{hayden:stein,ik:bennequin}) the Bennequin-Eliashberg inequality is indeed sharp for a strongly quasipositive transverse link. 

If $(Y,\xi)$ has a symplectic filling $(X,\omega)$, then Legendrian ribbons can also be used to understand the minimal genus of properly embedded ``slice'' surfaces in $X$ with boundary $L \subset Y$. To this end, recall the relative version of the Symplectic Thom Conjecture \cite{gk:caps}: If $\Sigma$ is a symplectic surface in $(X,\omega)$ such that $\partial \Sigma$ is a transverse link in $(Y,\xi)$, then $\Sigma$ is genus-minimizing in its relative homology class. We can apply this to a Legendrian ribbon by pushing its interior into a collar neighborhood of $X$ to produce a properly embedded symplectic surface (by, for example, combining Lemma~5.1 and Example~4.3 of \cite{hayden:stein}).

\begin{proposition}\label{prop:genus}
Let $(X,\omega)$ be a convex symplectic filling of $(Y,\xi)$. If $R$ is a Legendrian ribbon in $Y$ with transverse boundary $L$, then $R$ is genus-minimizing in its relative homology class in $H_2(X,L)$. 
\hfill $\square$
\end{proposition}

In particular, if $X$ has vanishing second homology, then a nontrivial strongly quasipositive knot in $Y$ cannot bound a slice disk in $X$. This generalizes a result of Rudolph \cite{rudolph:qp-obstruction} for strongly quasipositive links in $S^3$.

\subsection{Strongly quasipositive satellites} We begin by proving Theorem~\ref{thm:satellite}, which says that if $J \subset S^1 \times D^2$ is a strongly quasipositive braid and $K \subset (Y,\xi)$ is a strongly quasipositive transverse link, then the transverse satellite $J(K) \subset (Y,\xi)$ is also a strongly quasipositive transverse link.

\begin{figure}[b] \center
\includegraphics[width=.85\linewidth]{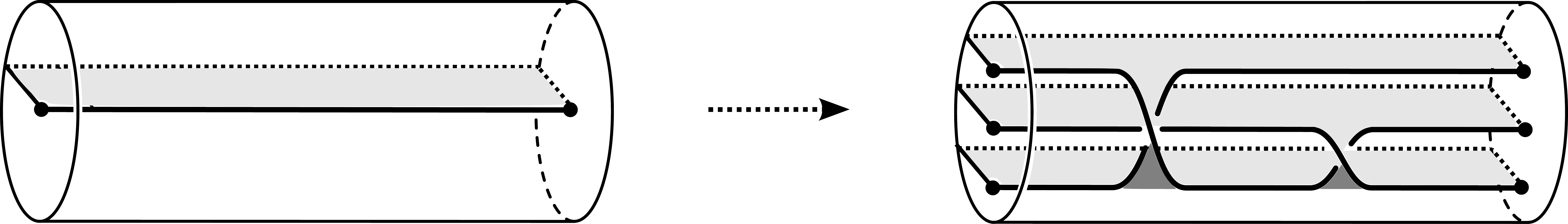} 
\caption{Taking a satellite whose pattern is a strongly quasipositive braid.}
 \label{fig:satellite}
\end{figure}

\begin{proof}[Proof of Theorem~\ref{thm:satellite}]
Fix a Legendrian ribbon $R$ bounded by $K$ in $(Y,\xi)$. The satellite is constructed using a contactomorphism between neighborhoods $V=N(U)$ and $N(K)$, where $U$ can be taken to be the standard transverse 1-braid bounding a disk $D$ perpendicular to the $z$-axis in $\rr^3$ with its rotationally symmetric contact structure.  Moreover, we can assume that the contactomorphism carries $D \cap N(U)$ to $R \cap N(K)$.

Choose a Bennequin surface $F$ for $J$ constructed from $n$ parallel copies of $D$ joined by positively twisted bands as in Figure~\ref{fig:satellite}. Letting $F_0$ denote $F \cap N(U)$, we note that the embedding $F_0 \subset N(U) \hookrightarrow N(K)$ respects the characteristic foliation. We can take a Reeb vector field for $\xi$ that is normal to $R$ and use it to obtain $n$ parallel copies of $R$, which we denote $R'$. We can then glue $R'$ to the copy of $F_0$ in $N(K)$ to obtain a Seifert surface for $J(K)$. The characteristic foliation this surface can be taken to satisfy the hypotheses of the ribbon flexibility lemma (Lemma~\ref{lem:recognize}), so we conclude that $J(K)$ bounds a Legendrian ribbon and is therefore strongly quasipositive.
\end{proof}

\begin{remark}
The argument above  extends to the case where $J$ is not necessarily a strongly quasipositive braid in $V$ but rather a strongly quasipositive link embedded in $V$ so that it bounds a ``relative'' Legendrian ribbon. More precisely, we need only require that $J$ cobound a surface $F_0$ with some number of transverse longitudes in $\partial V$ such that the characteristic foliation on $F_0$ is Morse-Smale, contains no negative singularities or attracting closed leaves, and points outward along $J$ and inward along the longitudes comprising $\partial F_0 \cap V$.
\end{remark}

We can use Theorem~\ref{thm:satellite} and the Bennequin-Eliashberg inequality to determine which cables of strongly quasipositive links are strongly quasipositive.

\begin{corollary}\label{cor:cables} Let $K\subset (Y,\xi)$ be a strongly quasipositive transverse link and let $(p,q)$ be a pair of  integers with $p\geq 1$. The transverse $(p,q)$-cable of $K$ is strongly quasipositive if and only if $q$ is non-negative.\end{corollary}

\begin{proof}
Denote the $(p,q)$-cable of $K$ by $K_{(p,q)}$. We can construct a Seifert surface $S$ for $K_{(p,q)}$ as in the proof of Theorem~\ref{thm:satellite}, where the bands may be negatively twisted if $q$ is negative. By \cite[\S21, Satz 1]{schubert}, this is a minimal genus Seifert surface for $K_{(p,q)}$. When $q$ is non-negative, the pattern $T_{(p,q)}$ is a strongly quasipositive braid, so $K_{(p,q)}$ is strongly quasipositive by Theorem~\ref{thm:satellite}. In particular, the self-linking number of $K_{(p,q)}$ is $-\chi(S)$. On the other hand, if $q$ is negative, then the maximal Euler characteristic for $K_{(p,q)}$ is unchanged but the self-linking number is decreased. It follows that the Bennequin-Eliashberg bound fails to be sharp, so $K_{(p,q)}$ cannot be strongly quasipositive if $q$ is negative.
\end{proof}

\begin{remark}
The above corollary generalizes a result of Hedden \cite[Corollary~1.3]{hedden:pos} that determines the strong quasipositivity of certain iterated torus knots in $S^3$.
\end{remark}

The next two examples produce strongly quasipositive links using satellites whose patterns already bound Legendrian ribbons in $S^1 \times D^2$, generalizing constructions due to Rudolph \cite{rudolph:constructions2,rudolph:qp-obstruction}.

\begin{example}[Quasipositive annuli]\label{ex:annuli} Given a Legendrian knot $\Lambda$ in $(Y,\xi)$, its Legendrian ribbon is an annulus whose boundary is a strongly quasipositive transverse link. Following Rudolph, we refer to such a surface as a \emph{quasipositive annulus}.\end{example}

\begin{example}[Whitehead doubles]\label{ex:doubles}
With $\Lambda$ still denoting a Legendrian knot and $R_\Lambda$ its ribbon, we can plumb together $R_\Lambda$ with the ribbon of a small Legendrian unknot $U$ with maximal Thurston-Bennequin number $\tb(U)=-1$. The boundary of the resulting ribbon is a twisted Whitehead double of $\Lambda$, where the twisting is determined by the contact framing. In particular, if $\Lambda$ is nullhomologous, then we have constructed the $\tb(\Lambda)$-twisted Whitehead double of $\Lambda$.
\end{example}

As evidenced by Examples~\ref{ex:annuli} and \ref{ex:doubles}, the companion to a strongly quasipositive satellite knot is not necessarily strongly quasipositive. 

\begin{example}[Cables of rational bindings] For a different flavor of example, we can consider a rationally nullhomologous knot $K \subset Y$ whose complement fibers over the circle --- that is, $K$ is the binding of a rational open book. Since $K$ is nontrivial in $H_1(Y;\mathbb{Z})$, it cannot be strongly quasipositive with respect to any (integral) open book. On the other hand, $K$ has positive cables that are genuinely fibered and support the same contact structure; see \cite{bevhm}. It follows that these cables can be represented by strongly quasipositive braids with respect to the open books they define.
\end{example}

On the other hand, all of the above examples of strongly quasipositive satellite knots have strongly quasipositive patterns. This motivates the following question:

\begin{question}\label{ques:satellite}
If a satellite knot $J(K)$ is strongly quasipositive with respect to a tight contact manifold $(Y,\xi)$, then must the pattern knot $J$ be strongly quasipositive in $(S^3,\xi_\st)$?
\end{question}

We address a special case of this question further below in \S\ref{subsec:fibered}.

\subsection{Strongly quasipositive fibered links} \label{subsec:fibered}

We can address special cases of Conjecture~\ref{conj:bennequin} and Question~\ref{ques:satellite} for fibered links. First, note that the characteristic foliation on a page of an open book can be assumed to be Morse-Smale with no closed leaves and only positive singularities. Applying Lemma~\ref{lem:recognize} and Theorem~\ref{thm:sqp-ribbon}, it follows that the binding of an open book is always transversely isotopic to a strongly quasipositive braid with respect to any open book supporting the same contact structure. The statement of Corollary~\ref{cor:fibered} given in \S\ref{sec:intro} now follows from a result of Etnyre and Van Horn-Morris:

\begin{theorem}[\cite{E-VHM:fibered}] Let $K$ be a fibered link in a tight contact manifold $(Y,\xi)$ with zero Giroux torsion. Then the contact structure $\xi_K$ supported by $K$ is isotopic to $\xi$ if and only if  $K$ achieves equality in the Bennequin bound.
\end{theorem}

We can apply these ideas to \cite[Examples 5.6-5.7]{ik:bennequin}. Here, Ito and Kawamuro observe that any $(N,1)$-cable $K_{(N,1)}$ of the binding $K$ of an open book can be realized by a one-stranded braid that achieves Bennequin's bound but is not strongly quasipositive as a braid. As they remark, this does not contradict Conjecture~\ref{conj:bennequin} because $K_{(N,1)}$ may become strongly quasipositive after transverse isotopy. Indeed, Theorem~\ref{thm:sqp-ribbon} and Corollary~\ref{cor:cables} imply that $K_{(N,1)}$ is transversely isotopic to a strongly quasipositive braid with respect to the original open book.

We now address a special case of Question~\ref{ques:satellite}. Before stating the result, we recall that an oriented link $J$ in the solid torus $V$ is said to be \textbf{fibered in the solid torus} if its exterior fibers over $S^1$ in such a way that the boundary of each fiber $F$ consists of a single longitude on $\partial N(J)$ and some number of longitudes on $\partial V$. This fibration constitutes a \textbf{relative open book decomposition} of the solid torus as defined in \cite{vhm:thesis} (see also \cite{bevhm}). We say that a contact structure $\xi$ is compatible with the relative open book if there is a contact form $\alpha$ for $\xi$ such that $J$ is a positively transverse link, $d\alpha$ is a positive area form on each page of the relative open book, and the characteristic foliation of $\xi$ on $\partial V$ agrees with the foliation defined by the boundaries of the fibers. However, it is convenient to relax the boundary condition and demand only that the characteristic foliation on $\partial V$ be linear and that the Reeb vector field for $\alpha$ preserve $\partial V$ and be positively transverse to both the characteristic foliation and the foliation induced by the fibers. It is easy to see that if $\xi$ and $\xi'$ are contact structures compatible with a relative open book decomposition of the solid torus in the strict and relaxed sense, respectively, then $\xi$ is tight if and only if $\xi'$ is tight.

\begin{definition}
A link $J$ in the solid torus $V= S^1 \times D^2$ is  \textbf{tight fibered in the solid torus} if $(V,J)$ is fibered and the associated relative open book for $V$ supports a tight contact structure.
\end{definition}

\begin{lemma}\label{lem:tight}
A tight fibered link in the solid torus is a tight fibered link in $S^3$.
\end{lemma}

\begin{proof}
Let $(J,\pi)$ denote a the relative open book supporting a tight contact structure $\xi$ on the solid torus $V$. We can produce an open book $(J,\pi')$ for $S^3$ with binding $J$ by capping off each fiber $\pi^{-1}(\theta)$ with disks attached along $\pi^{-1}(\theta)\cap \partial V$. These disks sweep out a complementary solid torus $W$ to $V$ in $S^3$. We can equip $W$ with a contact structure in the obvious way so that it glues together with $(V,\xi)$ to produce a contact structure $\xi'$ on $S^3$ supported by the open book $(J,\pi')$; see \cite[Proposition 3.0.7]{vhm:thesis}. 

For the sake of contradiction, suppose that $\xi'$ is overtwisted. Let $U$ denote the braided transverse unknot in $(J,\pi')$ that forms the core of $W$. By \cite[Corollary 2.3]{etnyre:overtwisted}, the complement of any transverse unknot in an overtwisted contact structure contains an overtwisted disk. (That is, all transverse unknots in an overtwisted contact manifold are ``loose''.) It follows that there is an overtwisted disk in the exterior of some smaller standard neighborhood of $U$. But the contact structure on the exterior of any such neighborhood of $U$ can also be realized as a contact structure on $V$ supported by $(J,\pi)$. This contradicts the hypothesis that $J$ is tight fibered in the solid torus, so we conclude that $\xi'$ is in fact tight, i.e. $J$ is tight fibered in $S^3$.
\end{proof}

We now show that the answer to Question~\ref{ques:satellite} is \emph{yes} for fibered satellite knots in a tight contact structure. In fact, we can say more:

\begin{theorem}\label{thm:fibered}
If $J(K)$ is a fibered satellite knot in $Y$ supporting a tight contact structure, then the pattern $J \subset V$ is a tight fibered knot in the solid torus (and in $S^3$) and the companion $K \subset Y$ is a rationally fibered knot, where the interior of the rational fiber surface is an open Legendrian ribbon. 
\end{theorem}

\begin{proof}
We begin by mirroring an argument from \cite[\S1.4]{eisenbud-neumann} and \cite[Expos{\'e} 14]{FLP}.  The incompressible torus $T = \partial N(K)$ in the exterior of $J(K)$ can be isotoped so that it meets the pages of the open book transversely. Let $F$ denote a fixed page. Using a flow on the exterior of $J(K)$ that preserves $T$, we can find a monodromy representative $\varphi: F\to F$ that fixes (setwise) the collection of homotopically nontrivial curves $\Gamma=T \cap F$. Since $T$ separates $Y$, the curves $\Gamma$ separate $F$ into subsurfaces $F_1$ and $F_2$. Here we take $F_1$ to be the connected subsurface containing $\partial F$ and $F_2$ to be the (possibly disconnected) subsurface in the interior of $F$. The monodromy must fix each of $F_1$ and $F_2$ setwise;  we denote by $\varphi_1$ and $\varphi_2$ the restriction of $\varphi$ to $F_1$ and $F_2$, respectively. For suitable choices in the Thurston-Winkelnkemper construction (including choosing a 1-form $\beta$ on $F$ that is $\varphi$-invariant near $\Gamma$), we can find a contact form $\alpha$ compatible with $(F,\varphi)$ such that the characteristic foliation on each page is transverse to $\Gamma$ and points out of the subsurface $F_2$. (This would no longer be possible if there were connected components of $F_1$ that did not meet $\partial F$.) We can also assume that the Reeb vector field preserves $T$ and is positively transverse to the foliations on $T$ induced by the contact structure and by the pages of the open book.

The desired conclusions now follow easily from this setup: First, we see that $(F_1,\varphi_1)$ is an abstract relative open book for the solid torus $N(K) \cong V$ with binding $J \subset V$. By construction, $(F_1,\varphi_1)$ supports the contact structure obtained by restricting $\alpha$, which must be tight because $\alpha$ defines a tight contact structure on $Y$. By Lemma~\ref{lem:tight}, it follows that $J$ is a tight fibered knot when viewed in $S^3$. Second, we have exhibited the exterior of $K$ in $Y$ as a mapping torus of $\varphi_2: F_2\to F_2$, and the characteristic foliation on each fiber can be assumed to satisfy Lemma~\ref{lem:recognize}. 
\end{proof}

Restricting our attention to knots in $S^3$, we obtain:

\begin{corollary}\label{cor:satellite}
If the satellite knot $J(K)$ in $S^3$ is a fibered strongly quasipositive knot, then both the pattern $J$ and companion $K$ are fibered strongly quasipositive knots.
\end{corollary}

\begin{proof}
The first part of the statement follows from Theorem~\ref{thm:fibered} and Hedden's characterization of tight fibered knots in $S^3$ as fibered strongly quasipositive knots  \cite[Proposition~2.1]{hedden:pos}. As for the companion knot, we recall that any rationally fibered knot $K$ in $S^3$ must in fact be genuinely fibered. Moreover, each rational fiber is a disjoint union of genuine fiber surfaces. Choosing any individual connected component of the rational fiber surface for $K$ guaranteed by Theorem~\ref{thm:fibered}, we obtain a Seifert surface for a push-off transversely isotopic to $K$ that is isotopic to a Legendrian ribbon. It follows that $K$ is a strongly quasipositive fibered knot.
\end{proof}

\begin{remark}
See \cite{diaz:thesis} for more discussion of braided satellites and quasipositivity.
\end{remark}

\subsection{Braid operations and quasipositivity}\label{subsec:destab} Our final application of these ideas concerns the notion of \emph{Stein quasipositive braids} in a Stein-fillable open book.  For the sake of brevity here, we refer the reader to \cite{hayden:stein} for a precise definition. In $(D^2,\id)$, these braids are  precisely Rudolph's quasipositive braids.

\begin{figure}[b] \center
\def\svgwidth{.6\linewidth}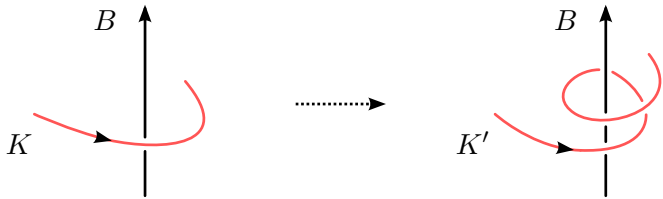 
\caption{A local depiction of positive Markov stabilization near the binding $B$.}
 \label{fig:stab}
\end{figure}

By the transverse Markov theorem (\cite{pavelescu:thesis}), any transverse isotopy between transverse braids in an open book can be realized by a sequence of positive Markov stabilizations and destabilizations; see Figure~\ref{fig:stab} for a local depiction of a positive Markov stabilization. As observed in \cite{hayden:stein}, it follows immediately from definitions that positive Markov stabilization preserves Stein quasipositivity at the level of braid isotopy. However, the converse is more subtle. In the classical case, we have:

\begin{theorem}[{Orevkov \cite{orevkov:markov}}]
Positive Markov destabilization preserves quasipositivity for braids in the open book $(D^2,\id)$ for $S^3$.
\end{theorem}

It follows that if $K$ and $K'$ are transversely isotopic braids in $(D^2,\id)$, then $K$ is a Stein quasipositive braid if and only if $K'$ is a Stein quasipositive braid. However, this fails to be true even in the next-simplest cases:

\begin{proposition}
If $(F,\varphi)$ is a Stein-fillable open book for $Y=\#^k S^1 \times S^2$ where $F \neq D^2$ has connected boundary, then positive Markov destabilization fails to preserve Stein quasipositivity of braids. 
\end{proposition}

\begin{proof}
The binding $K=\partial F$ is transversely isotopic to a 1-braid with respect to $(F,\varphi)$, viewed as a (1,1)-cable of $K$ with respect to the framing determined by the fiber $F$. It suffices to show that this braid is not Stein quasipositive but becomes so after some number of positive stabilizations.

For the latter claim, we note that any strongly quasipositive braid with respect to a Stein-fillable open book is necessarily Stein quasipositive. As discussed above, the binding of an open book admits a strongly quasipositive braid representative with respect to the open book itself, so we know that $K$ has a Stein quasipositive braid representative with respect to the Stein-fillable open book $(F,\varphi)$. By the transverse Markov theorem, the 1-braid and this Stein quasipositive braid have a common positive Markov stabilization. Since any positive Markov stabilization of a Stein quasipositive braid is again Stein quasipositive, we see that the 1-braid itself becomes Stein quasipositive after some number of positive Markov stabilizations.

To conclude the proof, we suppose for the sake of contradiction that the 1-braid representative of $K$ is indeed Stein quasipositive.  If the monodromy $\varphi$ is the identity, then a Stein quasipositive 1-braid must necessarily bound an embedded disk $D$ in $(F,\id)$. Since $(F,\id)$ naturally arises as the boundary of a subcritical Stein domain $X_0$, we can push  we can the interior of $D$ into a collar neighborhood of $X_0$. Otherwise, $\varphi$ consists of positive Dehn twists performed along homologically nontrivial simple closed curves in $F$. These correspond to modifying $(F,\id)$ by performing Dehn surgery along Legendrian curves in $F$ with framing $-1$ relative to the page framing. These surgeries can be viewed as the result of attaching Stein 2-handles to $X_0$ to produce $Y$ with the open book $(F,\varphi)$. Since the 2-handles are attached away from the 1-braid and thus away from $D$, we obtain a disk in a Stein filling $X$ of $Y$ bounded by the original 1-braid. But this leads to a contradiction of Proposition~\ref{prop:genus} (and the observation that follows it): Any Stein filling $X$ of $Y=\#^k S^1 \times S^2$ is diffeomorphic to $\natural^k S^1 \times B^3$ (see, e.g., \cite{ec:book}), which has vanishing second homology. It follows that a nontrivial strongly quasipositive knot in $(Y,\xi)$ cannot bound a slice disk in $X$, so we conclude that the 1-braid representative of $K$ cannot be Stein quasipositive.
\end{proof}

\bibliographystyle{alpha}
\bibliography{biblio}

\newcommand{\etalchar}[1]{$^{#1}$}
\begin{thebibliography}{BEVHM12}

\bibitem[Avd13]{avdek}
Russell Avdek.
\newblock Contact surgery and supporting open books.
\newblock {\em Algebr. Geom. Topol.}, 13(3):1613--1660, 2013.

\bibitem[Bak16]{baker:tight-fibered}
Kenneth~L. Baker.
\newblock A note on the concordance of fibered knots.
\newblock {\em J. Topol.}, 9(1):1--4, 2016.

\bibitem[BCV09]{bcv:ribbons}
Sebastian Baader, Kai Cieliebak, and Thomas Vogel.
\newblock Legendrian ribbons in overtwisted contact structures.
\newblock {\em J. Knot Theory Ramifications}, 18(4):523--529, 2009.

\bibitem[BEH{\etalchar{+}}15]{square}
I.~Baykur, J.~Etnyre, M.~Hedden, K.~Kawamuro, and J.~Van Horn-Morris.
\newblock Contact and symplectic geometry and the mapping class groups.
\newblock Official report of the 2nd square meeting, American Institute of
  Mathematics, July 2015.

\bibitem[Ben83]{bennequin}
D.~Bennequin.
\newblock Entrelacements et equations de {P}faff.
\newblock {\em Asterisque}, 107--108:87--161, 1983.

\bibitem[BEVHM12]{bevhm}
Kenneth~L. Baker, John~B. Etnyre, and Jeremy Van Horn-Morris.
\newblock Cabling, contact structures and mapping class monoids.
\newblock {\em J. Differential Geom.}, 90(1):1--80, 2012.

\bibitem[BI09]{bi:leg-qp}
Sebastian Baader and Masaharu Ishikawa.
\newblock Legendrian graphs and quasipositive diagrams.
\newblock {\em Ann. Fac. Sci. Toulouse Math. (6)}, 18(2):285--305, 2009.

\bibitem[CE12]{ec:book}
Kai Cieliebak and Yakov Eliashberg.
\newblock {\em From {S}tein to {W}einstein and back}, volume~59 of {\em
  American Mathematical Society Colloquium Publications}.
\newblock American Mathematical Society, Providence, RI, 2012.
\newblock Symplectic geometry of affine complex manifolds.

\bibitem[Cro98]{cromwell:survey}
Peter~R. Cromwell.
\newblock Arc presentations of knots and links.
\newblock In {\em Knot theory ({W}arsaw, 1995)}, volume~42 of {\em Banach
  Center Publ.}, pages 57--64. Polish Acad. Sci. Inst. Math., Warsaw, 1998.

\bibitem[Dia]{diaz:thesis}
Alan Diaz.
\newblock {P}h{D} thesis, {G}eorgia {I}nstitute of {T}echnology.
\newblock In preparation.

\bibitem[EF09]{ef:trivial}
Yakov Eliashberg and Maia Fraser.
\newblock Topologically trivial legendrian knots.
\newblock {\em The Journal of Symplectic Geometry}, 7(2):77, 2009.

\bibitem[Eli92]{yasha:20yrs}
Ya. Eliashberg.
\newblock Contact $3$-manifolds twenty years since {J}. {M}artinet's work.
\newblock {\em Ann. Inst. Fourier (Grenoble)}, 42(1-2):165--192, 1992.

\bibitem[EN85]{eisenbud-neumann}
David Eisenbud and Walter Neumann.
\newblock {\em Three-dimensional link theory and invariants of plane curve
  singularities}, volume 110 of {\em Annals of Mathematics Studies}.
\newblock Princeton University Press, Princeton, NJ, 1985.

\bibitem[Etn06]{etnyre:obd}
John~B. Etnyre.
\newblock Lectures on open book decompositions and contact structures.
\newblock In {\em Floer homology, gauge theory, and low-dimensional topology},
  volume~5 of {\em Clay Math. Proc.}, pages 103--141. Amer. Math. Soc.,
  Providence, RI, 2006.

\bibitem[Etn13]{etnyre:overtwisted}
John~B. Etnyre.
\newblock On knots in overtwisted contact structures.
\newblock {\em Quantum Topol.}, 4(3):229--264, 2013.

\bibitem[EVHM11]{E-VHM:fibered}
John~B. Etnyre and Jeremy Van Horn-Morris.
\newblock Fibered transverse knots and the {B}ennequin bound.
\newblock {\em Int. Math. Res. Not. IMRN}, (7):1483--1509, 2011.

\bibitem[FLP12]{FLP}
Albert Fathi, Fran\c{c}ois Laudenbach, and Valentin Po\'enaru.
\newblock {\em Thurston's work on surfaces}, volume~48 of {\em Mathematical
  Notes}.
\newblock Princeton University Press, Princeton, NJ, 2012.
\newblock Translated from the 1979 French original by Djun M. Kim and Dan
  Margalit.

\bibitem[Gir02]{giroux:geometry}
Emmanuel Giroux.
\newblock G\'eom\'etrie de contact: de la dimension trois vers les dimensions
  sup\'erieures.
\newblock In {\em Proceedings of the {I}nternational {C}ongress of
  {M}athematicians, {V}ol. {II} ({B}eijing, 2002)}, pages 405--414. Higher Ed.
  Press, Beijing, 2002.

\bibitem[GK16]{gk:caps}
Siddhartha Gadgil and Dheeraj Kulkarni.
\newblock Relative symplectic caps, 4-genus and fibered knots.
\newblock {\em Proc. Indian Acad. Sci. Math. Sci.}, 126(2):261--275, 2016.

\bibitem[GL]{g-l:morse}
David Gay and Joan~E. Licata.
\newblock Morse structures on open books.
\newblock {\em Trans. Amer. Math. Soc.}
\newblock To appear. (Available as arXiv:1508.05307).

\bibitem[Hay17]{hayden:stein}
Kyle Hayden.
\newblock Quasipositive links and {S}tein surfaces.
\newblock Available as arXiv:1703.10150, March 2017.

\bibitem[Hed]{hedden:subcritical}
Matthew Hedden.
\newblock Knot theory and algebraic curves in subcritical {S}tein domains.
\newblock In preparation.

\bibitem[Hed10]{hedden:pos}
Matthew Hedden.
\newblock Notions of positivity and the {O}zsv\'ath-{S}zab\'o concordance
  invariant.
\newblock {\em J. Knot Theory Ramifications}, 19(5):617--629, 2010.

\bibitem[IK14]{ito-kawamuro:open}
Tetsuya Ito and Keiko Kawamuro.
\newblock Open book foliation.
\newblock {\em Geom. Topol.}, 18(3):1581--1634, 2014.

\bibitem[IK17]{ik:bennequin}
Tetsuya Ito and Keiko Kawamuro.
\newblock The defect of {B}ennequin-{E}liashberg inequality and {B}ennequin
  surfaces.
\newblock Available as arXiv:1703.09322, March 2017.

\bibitem[Ni07]{ni:fibered}
Yi~Ni.
\newblock Knot {F}loer homology detects fibred knots.
\newblock {\em Invent. Math.}, 170(3):577--608, 2007.

\bibitem[Ore00]{orevkov:markov}
Stepan~Yu Orevkov.
\newblock Markov moves for quasipositive braids.
\newblock {\em C. R. Acad. Sci. Paris S\'er. I Math.}, 331(7):557--562, 2000.

\bibitem[OS03]{os:markov}
S.~Yu. Orevkov and V.~V. Shevchishin.
\newblock Markov theorem for transversal links.
\newblock {\em J. Knot Theory Ramifications}, 12(7):905--913, 2003.

\bibitem[Pav08]{pavelescu:thesis}
Elena Pavelescu.
\newblock {\em Braids and open book decompositions}.
\newblock PhD thesis, University of Pennsylvania, 2008.

\bibitem[Rud76]{rudolph:concordance}
Lee Rudolph.
\newblock How independent are the knot-cobordism classes of links of plane
  curve singularities?
\newblock {\em Notices Amer. Math. Soc.}, 23(410), 1976.

\bibitem[Rud84]{rudolph:constructions2}
Lee Rudolph.
\newblock Constructions of quasipositive knots and links. {II}.
\newblock In {\em Four-manifold theory ({D}urham, {N}.{H}., 1982)}, volume~35
  of {\em Contemp. Math.}, pages 485--491. Amer. Math. Soc., Providence, RI,
  1984.

\bibitem[Rud90]{rudolph:kauffman-bound}
L.~Rudolph.
\newblock A congruence between link polynomials.
\newblock {\em Math. Proc. Cambridge Philos. Soc.}, 107(2):319--327, 1990.

\bibitem[Rud93]{rudolph:qp-obstruction}
L.~Rudolph.
\newblock Quasipositivity as an obstruction to sliceness.
\newblock {\em Bull. Amer. Math. Soc. (N.S.)}, 29(1):51--59, 1993.

\bibitem[Sch53]{schubert}
Horst Schubert.
\newblock Knoten und {V}ollringe.
\newblock {\em Acta Math.}, 90:131--286, 1953.

\bibitem[TW75]{tw:existence}
W.~P. Thurston and H.~E. Winkelnkemper.
\newblock On the existence of contact forms.
\newblock {\em Proc. Amer. Math. Soc.}, 52:345--347, 1975.

\bibitem[VHM07]{vhm:thesis}
Jeremy Van Horn-Morris.
\newblock {\em Constructions of Open Book Decompositions}.
\newblock PhD thesis, University of Texas, Austin, 2007.

\end{thebibliography}
\end{document}